\documentclass[a4wide,11pt]{article}
\usepackage[utf8]{inputenc}
\usepackage{amsmath,amsfonts,amssymb}
\usepackage{amsthm}
\usepackage{array,colortbl,xcolor}
\usepackage{hyperref}

\usepackage{orcidlink}

\newtheorem{theorem}{Theorem}[section]

\newtheorem{lemma}[theorem]{Lemma}
\newtheorem{remark}{Remark}

\newcommand{\norm}[2][]{\left\lVert#2\right\rVert_{#1}}

\usepackage{color}

\title{The turnpike property for mean-field optimal control problems}

\date{ }

\author{Martin Gugat\thanks{Chair in Dynamics, Control and  Numerics (Alexander von Humboldt-Professorship),
		Department of Data Science, Friedrich-Alexander Universit\"at Erlangen-N\"urnberg (FAU),
		Cauerstr. 11, 91058 Erlangen, Germany, (\href{mailto:martin.gugat@fau.de}{martin.gugat@fau.de} \orcidlink{0000-0002-5281-110X})}, 
	\ Michael Herty\thanks{Chair in Numerical Analysis, IGPM, RWTH Aachen University, Templergraben, 55, D-52062 Aachen, Germany (\href{mailto:herty@igpm.rwth-aachen.de}{herty@igpm.rwth-aachen.de} \orcidlink{0000-0002-6262-2927})}, 
	\ and Chiara Segala\thanks{IGPM, RWTH Aachen University, Templergraben, 55, D-52062 Aachen, Germany (\href{mailto:segala@igpm.rwth-aachen.de}{segala@igpm.rwth-aachen.de} \orcidlink{0000-0002-6480-3772})}
}

\begin{document}
	\maketitle
	\begin{abstract}
  We study the turnpike phenomenon for  optimal control problems
with mean field dynamics that are obtained as the limit $N\rightarrow \infty$ of systems
governed by a large number $N$ of ordinary differential equations.
We show that the optimal control problems with with large time horizons
give rise to a turnpike structure of the optimal state and the optimal control.
For the proof, we use the fact that
the turnpike structure for the problems on the level of ordinary differential equations
is preserved under the corresponding mean-field limit.

	\end{abstract}
	
	{\bf Keywords.} 
	\par Turnpike property, mean-field limit, optimal control 
	
	{\bf AMS Classification.}  
	\par 93C20, 35L04, 35Q89, 49N10

\section{Introduction}\label{sec:intro}
		Over the last few years, there has been an increased level of activity on the study of collective behavior phenomena from a multiscale modeling perspective. Classical examples in socio-economy, biology, and robotics are given by the interactions between self-propelled particles, such as animals and robots, see e.g. \cite{bellomo20review, MR2974186, MR2165531, MR3119732, MR2580958,Giselle}. Those particles interact according to a nonlinear model encoding various social rules for example attraction, repulsion, and alignment.
		
		It is of great relevance for applications in the study of the impact of control inputs in such complex systems. Results in this direction allow the design of optimized actions such as collision-avoidance protocols for swarm robotics \cite{CKPP19}, pedestrian evacuation in crowd dynamics \cite{MR3542027}, the quantification of interventions in traffic management \cite{MR3948232} or in opinion dynamics \cite{AHP15,Garnier}. From a mathematical point of view, a multiagent control problem is described by minimization of an integral objective functional subject to a constraint that is the complex dynamic depicted by a system of ordinary differential equations (ODE). 
	
	The  formulation of an interacting particle system at a microscopic level  requires the study of large-scale systems of agents (or particles) and it requires a considerable effort both from a theoretical and numerical point of view. We may consider a different level of description, that is the derivation of mesoscopic or mean-field approximations of the original dynamic. Here,  the density of the particles is obtained as the number of particles tends to infinity \cite{MR2438213,1556-1801_2022_2_129,MR2740099,MR2425606,lasry2007mean,degond2014large, herty2019consistent, ha2009simple, cardaliaguet2010notes,JacobTotzeck2022}. Of particular interest is therefore the design of controls in the mean-field control approaches  \cite{MR3264236,MR4028474,FPR14,BORZI20204973}.
	
	In this paper, we focus on the turnpike phenomenon for mean-field optimal control problems. This topic has been studied recently for example in
	\cite{https://doi.org/10.48550/arxiv.2101.09965}, and it concerns relations between the solutions
	of dynamic optimal control problems with objective functionals of tracking type
	and the corresponding static optimal control problems. 
	The turnpike property states that  the distance between
	the dynamic and the static  optimal  solution is small, in particular, for large time intervals. Hence, it allows reducing the control effort by solving only static optimal control problems.

	An early to the turnpike property is  \cite{samuelson}, and in \cite{larsgruene}  an
	overview on discrete-time and continuous-time
	turnpike properties are given.
	The turnpike phenomenon for systems governed by
	ordinary differential equations has been studied also in detail in \cite{trelat2015turnpike,zbMATH07063084,geshkovski_zuazua_2022}.  
	measure and integral turnpike properties have been studied in
	\cite{zbMATH06903641}.
	A turnpike analysis for systems that are governed by semilinear partial
	differential equations is presented  in \cite{zbMATH07369283}, while the relation between the turnpike property and the receding-horizon method
	is investigated in \cite{MR4082481}. In \cite{manifoldTP} manifold turnpikes are also studied. Contrary to those works we are interested in the question if the turnpike property of a system persists in the limit of infinitely many ODEs and under which conditions such a turnpike property holds true on the mean-field level.
	
Next, we state the optimal control problem in detail. We consider the control of high-dimensional nonlinear dynamics accounting for the evolution of $N$ agents at the microscopic level and, as described for example in
	\cite{1556-1801_2022_2_129}, the mean-field dynamics given by
	a non-local transport equation for the
	density of particles at position $x \in \mathbb{R}^d$ and time $t \in \mathbb{R}^+$.
	The initial particle density $\mu^0(x)$ 
	 is given and the control action is
	modeled by an additive term in
	the partial differential equation (PDE). More specifically, we consider a
	PDE of the type
	\begin{equation}\label{eq:meso_dyn}
	\partial_t \mu(t,\, x) +
\partial_x \Bigl( \left( (P\ast \mu)(t,x) + u(t,x)\right) \, \mu(t,\, x) \Bigr) = 0, \qquad \mu(a, x) = \mu^0(x),
	\end{equation}
	where
	$\ast$ denotes the convolution operator, the function $P$ is given, and the real positive number $a$ is the initial time. 
	Dynamics of this type may also occur as nonlocal regularizations
	of balance laws, see e.g., \cite{COCLITE2021112370, doi:10.1137/20M1366654}.
	
	We consider an optimal control problem for a finite large time horizon, subjected to system \eqref{eq:meso_dyn}.
	The objective function that  we want to minimize depends both on the control and the state
	\[	
	\textbf J_{( a,\, b)}(\mu, \,u)=
	\int_a^b f( \mu(t,\,x),  \, u(t, x) ) \, dt,
	\]
	for a given real-valued function $f$ 
	\begin{equation}\label{eq:f_meso}
f(\mu, \,u) =
\int_{\mathbb{R}^d} 
\left(
L(x) + \Psi  (u(t,x))
\right) \, d \mu(t,\, x),
	\end{equation}
	and a time interval $[a,b]$ with $a<b$ real positive numbers. A particular example is given by 
	\begin{equation*}
		\partial_t \mu(t,\, x) +
		\partial_x \Bigl( \left( (\mathcal{H}\ast \mu)(t,x)  + u(t,x)\right) \, \mu(t,\, x) \Bigr) = 0,
	\end{equation*}
	where
	$ (\mathcal{H}\ast \mu)(t,x)  = \int_{\mathbb{R}^d}
	H(x-y)( y - x) \, d \mu(t, \, y) $ denotes a non-local integral operator, and $H$ a given continuous function. For this example, we assume that the integral objective function is the sum of
	a quadratic control cost and a tracking	term.
	The tracking term is the mean-field 
	limit of a microscopic term that aims all the particles to reach
	a constant consensus state $\bar{\psi}$.
	Hence, the 	size of  
	the difference 
	\[ \int_{\mathbb{R}^d} \norm{ x - \overline {\psi}}^2 \, d \mu(t, \, x) 	\]
	is minimized in a suitable norm,
	assuming that
	$1 = \int_{\mathbb{R}^d} 1 \, d \mu(t, \, x)$.
For a parameter $\gamma \geq 0$, we
consider $f = \hat{f}$
\[
\hat{f}( \mu,  \, u )=
\int_{\mathbb{R}^d} \left(
\norm{ x  - \overline 
	\psi }^2
+ \gamma \norm{ u(t,x) }^2	\right) \, d \mu(t,\, x) 
\,.
\]
Note that this is a particular case included in 
the cost functional from \cite{Fornasier_2014}. Therein, the minimization of a more general integral cost constrained by a PDE is considered. 
The analysis in \cite{Fornasier_2014} is applicable to our optimal control problem, in particular, the existence result for controls  provide in Theorem 4.7 but the turnpike property is not discussed.

The paper is organized as follows. In Section \ref{sec:ocp_micro} the general dynamic optimal control problem is defined at microscopic level.
Section \ref{sec:sol_meanfield} is devoted to the mean-field approximation of the microscopic dynamics and presents an existence result for the solution of the mesoscopic control problem.
In Section \ref{sec:dissipativity} we show that the
problem satisfies a strict dissipativity inequality both at the microscopic and mean-field level.
In Section \ref{sec:cheap}, we prove that a cheap control condition holds, we first discuss it
for the case with a finite number of particles
and then we extend it to the mean-field case. 
Finally, Section \ref{turnpike} uses the previous assumptions to prove the turnpike property with interior decay.	

\section{An optimal control problem for $N$ particles} \label{sec:ocp_micro}

Let natural numbers
$d$
and 
$N$  be given.
Define the state space 
\[X_N = ({\mathbb R}^d)^N.\]
%
Let  initial particle states 
$\psi^0\in X_N$
be given.
For 
$\psi_k(t) \in {\mathbb R}^d$  ($k\in \{1,...,N\}$)
as in \cite{zbMATH07510074}, 
we consider the system
with the initial conditions
$\psi_k(a) = \psi^0_k$   ($k\in \{1,...,N\}$).
Let a continuous function
\[
P: \, {\mathbb R}^d
\rightarrow  {\mathbb R}, \quad \text{with} \
P(0)=0,
\]
be given
that is bounded
with respect to the maximum norm.
For $k\in \{1,...,N\}$
the movement of the particles is 
governed by the
ordinary differential equations

\begin{equation}\label{eq:micro_dyn}
\begin{split}
\psi_k'(t) &=
(P \ast \mu_N)(\psi_k(t))
+
u_k(t)
\\
&= \frac{1}{N} \sum_{i=1}^{N} 
P(\psi_i(t) - \psi_k(t))
+
u_k(t), \quad \psi_k(a)= \psi^0_k,
\end{split}
\end{equation}
where $u_k(t) = u(t, \psi_k(t))$ and
\[
\mu_N(t,x)= \frac{1}{N} \sum_{i=1}^{N} \delta(x-\psi_i(t))
\]
is the empirical measure supported on the agents states.
We search for a control $u_k(t)$ 
that is a solution of the optimal control
problem where the cost functional
\begin{align}\label{eq:micro_J}
\mathcal J_{(N,a,b)}(\psi, u) = 
\int_a^b
f_N(\psi(t),u(t)) \, dt,
\end{align}
is minimized,  with
\begin{equation}\label{eq:f_N}
	\begin{split}
	f_N(\psi,u)  &=
	\int_{\mathbb{R}^d} \left(
	L(x) + \Psi  (u(t,x))
	\right) \, d \mu_N(t,\, x)
	\\
	&= \frac{1}{N} 
	\sum_{k=1}^N \Bigl(  L(\psi_k(t))
	 + \Psi (u_k(t))
	\Bigr),
		\end{split}
\end{equation}
where the optimization horizon $b-a$ expresses the time 
horizon 
along which we minimize the running cost. Thus our objective is a function of the state and control variables.
For $k\in \{1,...,N\}$, we define the static variables
$\psi^{(\sigma)}, \ u^{(\sigma)}$, for the state and control respectively.
Then with the initial data
$\psi_k(a) = \psi^{(\sigma)}$ 
and the control
$u^{(\sigma)}_k = u^{(\sigma)} = 0$
the system remains in the steady state
$ \psi_k(t)  =  \psi^{(\sigma)} $, for every $k\in \{1,...,N\}$.
The problem is similar to the problem that has
been considered in \cite{Fornasier_2014}. 

For real positive numbers $a, b$, with $b>a$
and  an initial state $\psi^0\in X_N$ 
we define the parametric optimization problem
\[\mathcal Q(N, \, a,\, b , \, \psi^0):\,
\min_u
\mathcal J_{(N,\, a,\, b)}(\psi, \,u)
\]
subject to \eqref{eq:micro_dyn}, and $\mathcal{V}(N, a,\,b,\,\psi^0)$ the optimal value of
$\mathcal Q(N, \, a,\, b, \, \psi^0)$.

\section{Existence of solutions in the mean-field limit} \label{sec:sol_meanfield}

The original formulation of the interacting particle system \eqref{eq:micro_dyn} is at microscopic level through a system of ODEs, but the study of microscopic model for a large system of individuals implies a considerable effort especially in numerical simulations, as models on real data may take into account very large number of interacting individuals. To reduce this complexity we can consider a more general level of description, that is the derivation of a mesoscopic approximation of the original dynamic. The basic idea is to analyse the density of particles, instead of focusing on the evolution of every single particle. Hence we will consider continuous models in order to simulate the collective behaviour in case of analysing systems with a large number of agents $N\gg 1$.
By passing to the mean-field limit $N \rightarrow \infty$ of the ODE system \eqref{eq:micro_dyn}, we obtain the PDE problem \eqref{eq:meso_dyn} which describes how the density of the particles $\mu = \mu(t,x)$ changes in time.

In order to prove the existence of a mean-field limit for the dynamics \eqref{eq:micro_dyn}
and the cost functional \eqref{eq:micro_J}
we consider the functions with the following properties:

\begin{itemize}
	\item[(P)] The function $P:\mathbb{R}^d \rightarrow \mathbb{R^d}$, with $P(0)=0$,
	is a locally Lipschitz function such that
	\[
	\Vert P(\psi) \Vert \leq C_P \Vert \psi \Vert, \quad \text{for all } \psi \in \mathbb{R}^d.
	\]
	
	\item[(L)] The function $L:\mathbb{R}^d \rightarrow [0, + \infty)$
	is a continuous function with respect to the topology generated by the Euclidean distance on $\mathbb{R}^d$. 
	
	\item[($\Psi $)] The function $\Psi :\mathbb{R}^d \rightarrow [0, + \infty)$, with $\Psi(0)=0$,
	is a non negative convex function and there exist $C_{\Psi} \geq 0$ and $1\leq q \leq + \infty$ such that
	\[
	\text{Lip} (\Psi , B(0,R)) \leq C_{\Psi} R^{q-1},
	\]
	for all $R>0$.

\end{itemize}
Considering these three assumptions for the functions $P,L,$ and $\Psi $, we can apply the existence Theorem 4.7 in \cite{Fornasier_2014}. In Theorem \ref{thm:fornasier} we indicate with $\mathcal{W}_1$ the Wasserstein distance between two probability measures $\mu$ and $\nu \in P_1(\mathbb{R}^d)$ as

\[
\mathcal{W}_1(\mu,\nu) := \inf_{\gamma \in \Gamma (\mu,\nu)} \int_{\mathbb{R}^d\times \mathbb{R}^d} \Vert x-y \Vert \ d\gamma(x,y),
\]
where $\Gamma(\mu,\nu)$ denotes the collection of all measures on $\mathbb{R}^d\times \mathbb{R}^d$ with marginals $\mu$ and $\nu$ on the first and second factors respectively.
\begin{theorem}\label{thm:fornasier}
	Let $\mu^0\in P_1(\mathbb{R}^d)$ be a given probability measure with compact support. 
	We assume that the sequence $(\mu_N^0)_{N\in \mathbb{N}}$ of empirical measures $\mu_N^0(x) = \frac{1}{N} \sum_{i=1}^N \delta(x-\psi_i^0)$ is such that $\lim_{N\rightarrow\infty} \mathcal{W}_1(\mu_N^0,\mu^0) = 0$. Let
	\[
	\mu_N(t,x) = \frac{1}{N} \sum_{i=1}^N \delta(x-\psi_i(t)),
	\]
	be supported on the phase space trajectories $\psi_i(t) \in \mathbb{R}^d$, for $i=1,\dots,N$, defining the solution of \eqref{eq:micro_dyn} in $[a,b]$
	with initial state $\psi(a) = \psi^0$. Then there exists a map $\mu \in P_1(\mathbb{R}^d)$ such that
	\begin{itemize}
		\item $\lim_{N\rightarrow\infty} \mathcal{W}_1(\mu_N(t),\mu(t)) = 0$ uniformly with respect to $t\in [a,b]$;
		\item $\mu$ is a weak equi-compactly supported solution of \eqref{eq:meso_dyn};
		\item regarding the cost functional \eqref{eq:micro_J}, the following limit holds:
		
		\begin{equation*}
			\begin{split}
				\lim_{N\rightarrow\infty} \
				&	\int_a^b
				\int_{\mathbb{R}^d} \left(
				L(x) + \Psi  (u(t,x))
				\right) \, d \mu_N(t,\, x) 
				\,dt \\
				=& \int_a^b
				\int_{\mathbb{R}^d} \left(
				L(x) + \Psi  (u(t,x))
				\right) \, d \mu(t,\, x) 
				\,dt .		
			\end{split}
		\end{equation*}

	\end{itemize}
	
\end{theorem}

Theorem \ref{thm:fornasier} holds for general $P,L,\Psi$ functions that satisfy the hypothesis $(P),(L),(\Psi)$. We can observe these assumptions are satisfied for the example we took into account in the introduction \ref{sec:intro}, where
\[
P  (\psi) := H(\psi) \psi, \qquad
L(\psi) := \Vert \psi - \overline {\psi} \Vert^2, \qquad
\Psi  (u) := \gamma \Vert u \Vert^2.
\]

We define the parametric mean-field optimization problem
\[\textbf{Q}( a,\, b , \, \mu^0):\,
\min_{u}
\textbf J_{( a,\, b)}(\mu, \,u)
\]
subject to \eqref{eq:meso_dyn}. We recall
the mean-field objective functional is
\begin{equation}
	\label{jinfty}
	\textbf J_{( a,\, b)}(\mu, \,u)  =
	\int_a^b
	\int_{\mathbb{R}^d} \left(
	L(x) + \Psi  (u(t,x))
	\right) \, d \mu(t,\, x) 
	\,dt 
	.
\end{equation}
We define the optimal value of the mean-field limit problem $ \textbf Q (a,\, b , \, \mu^0) $ as $ {\textbf{V}}(a, \,b,\,\mu^0)$. 
The existence of solutions for $ \textbf Q (a,\, b , \, \mu^0)$ is guaranteed by 
Theorem 5.1 in \cite{Fornasier_2014}.

\section{The strict dissipativity inequality}\label{sec:dissipativity}
In this section we assume that the optimal control
problem 
satisfies a strict dissipativity assumption.
We start considering the $N$-particles problem and
then we proceed with the mean-field limit formulation.

\subsection{The strict dissipativity inequality for the microscopic problem}
For any admissible pair $(\psi(\cdot), u(\cdot))$
and for all $\tau \in [a,\, b]$,
{we assume that 
the following
\emph{dissipativity inequality} holds:}
\begin{equation}
	\begin{split}
\label{dissipationinequalityN}
\int_{a}^\tau &
f_N( \psi(t),\, u(t) ) \, dt\\
&\geq
\int_{a}^\tau
\frac{1}{N}
\left( \|  \psi(t) - \psi^{(\sigma)}\|_N
+
\|  u(t)  - u^{(\sigma)} \|_N
\right)^2
\, dt
	\end{split}
\end{equation}
Here  $
\|z\|_N = 
\sqrt{ \sum_{k=1}^N \| z_k\|^2}$, $\norm{\cdot}$ is the usual Euclidean norm, and $f_N$ is the running cost in \eqref{eq:f_N}.
Given an initial state $\psi^0$, the problem $\mathcal Q(N, \, a,\, b, \, \psi^0)$, i.e. the minimization over $u$ of the cost functional $\mathcal J_{(N,a,b)}$ in \eqref{eq:micro_J}, is a strictly dissipative problem in $[a,b]$ at $(\psi^{(\sigma)},u^{(\sigma)})$.

\subsection{The strict dissipativity inequality in the mean-field limit}

We consider the following computation starting from \eqref{dissipationinequalityN}
\begin{equation}\label{eq:comp_diss_micro}
	\begin{split}
		\int_{a}^\tau &
		f_N( \psi(t),\, u(t) ) \, dt\\
		&\geq
		\int_{a}^\tau
		\frac{1}{N}
		\left( \|  \psi(t) - \psi^{(\sigma)}\|_N
		+
		\|  u(t)  - u^{(\sigma)} \|_N
		\right)^2
		\, dt
		\\
		&=
		\int_{a}^\tau
		\frac{1}{N}
		\left( \sqrt{ \sum_{k=1}^N \|  \psi_k(t) - \psi^{(\sigma)}\|^2}
		+
		\sqrt{ \sum_{k=1}^N \|  u_k(t)  - u^{(\sigma)} \|^2}
		\right)^2
		\, dt\\
		&=
		\int_{a}^\tau
		\frac{1}{N}
		\left( { \sum_{k=1}^N \|  \psi_k(t) - \psi^{(\sigma)}\|^2}
		+
		{ \sum_{k=1}^N \|  u_k(t)  - u^{(\sigma)} \|^2}
		\right)
		\, dt
		\\
		&\quad
		+ \ \int_{a}^\tau
		\frac{2}{N}
	\sqrt{ \sum_{k,l=1}^N \|  \psi_k(t) - \psi^{(\sigma)}\|^2\, \|  u_l(t)  - u^{(\sigma)} \|^2}
		\, dt\\
		&\geq
		\int_{a}^\tau
		\frac{1}{N}
		\left( { \sum_{k=1}^N \|  \psi_k(t) - \psi^{(\sigma)}\|^2}
		+
		{ \sum_{k=1}^N \|  u_k(t)  - u^{(\sigma)} \|^2}
		\right)
		\, dt.
	\end{split}
\end{equation}
Since all the quantities in \eqref{eq:comp_diss_micro} admit a  mean-field limit, we can consider the problem for $N\rightarrow\infty$ thanks to
Theorem \ref{thm:fornasier}
and state a dissipativity nequality in $[a,b]$ in terms of measures.
We have for all $\tau \in [a,\, b]$
\begin{equation}\label{eq:strict_diss_meso}
	\begin{split}
	\int_a^\tau  & f(\mu(t,x), \,u(t,x))  dt\\
&\geq
\int_a^\tau
\int_{\mathbb{R}^d} \left(
\| x  -  
\psi^{(\sigma)} \|^2
+ \| u(t,x)- u^{(\sigma)} \|^2	\right) \, d \mu(t,\, x) 
\,dt ,
	\end{split}
\end{equation}
where $f$ is the functional in 
\eqref{eq:f_meso}, and it is
the mean-field limit of the microscopic running cost \eqref{eq:f_N}.


\section{The cheap control condition}\label{sec:cheap}
For our analysis, a cheap control condition is
essential. 
It requires that the 
optimal values are bounded in terms of
the distance between the initial state 
and the desired static state. We first discuss this assumption 
for the case with a finite number of particles
and then 
extend it to the mean-field case.

\subsection{The cheap control condition for the microscopic problem}

In this section we show that the optimization  problem
$\mathcal Q(N, \, a,\, b , \, \psi^0)$ satisfies a cheap control
condition in the following sense:

{\em
	There exist a constant $\mathcal C_0>0$
	such that
	for all initial times
	$a$, all  initial states
	$\psi^0$
	and
	for all   terminal times
	$b> a $
	we have
	the inequality}
\begin{equation}
	\label{exactcontrol}
	\mathcal{V} (N,\, a,\,b,\,\psi^0)
	\leq
	\mathcal C_0 \,
	\frac{1}{N}\sum_{k=1}^N \|\psi^0 - \psi^{(\sigma)}\|
	.
\end{equation}
\begin{remark}
	The cheap control condition and the dissipativity inequality in
	\eqref{dissipationinequalityN}
	imply that
	$\mathcal Q(N, \, a,\, b , \, \psi^0)$ 
	has the \emph{integral turnpike property},
	which means that 
	for the corresponding optimal state/control $(\psi, u)$ pair we
	have
	\begin{equation}
		\label{integraltp}
		\int_a^b 
		\left( \|  \psi(t) - \psi^{(\sigma)}\|_N
		+
		\|  u(t)  - u^{(\sigma)} \|_N
		\right)^2
		\, dt 
	\leq 
	{ \mathcal C_0}
	\sum_{k=1}^N \|\psi^0 - \psi^{(\sigma)}\|.
\end{equation}
Since the right-hand side is independent of
$b-a$, the inequality
implies that the distance between the
dynamic and the static optimal state
and control is uniformly bounded with respect to the time horizon.
This implies in particular that this distance must be small on the larger part of the time interval for sufficiently
large time horizons.

\end{remark}
In order to prove \eqref{exactcontrol} we consider a
stabilizing feedback law that leads to exponential 
decay of
$ f_{N} (\psi(t), \, u(t))  $.
Let a feedback parameter $\beta  >0$ be given.
We define the control
\begin{equation} \label{eq_cheap_control}
u(t,\psi_k(t)) = 
\beta \left(  \psi^{(\sigma)} - \psi_k(t) \right) 
- \frac{1}{N} \sum_{l=1}^N P(\psi_l(t)- \psi_k(t)).
\end{equation}
Then for the solution of
the initial value problem 
with the initial states
$\psi^0_k$ at the time $a$ and the differential equations  
\[\psi_k'(t) =
\frac{1}{N} \sum_{l=1}^N P(\psi_l(t)- \psi_k(t)) 
+
u(t,\psi_k(t))\]
we have
\[\psi_k'(t) =
\beta \left( \psi^{(\sigma)}  - \psi_k(t) \right).
\]

\begin{lemma}\label{lem:micro}
Consider the additional local assumption on bounded level sets of $L$:
\begin{equation}\label{eq:L_ass}
L(\psi) 
\leq C_L \, \Vert \psi - \psi^{(\sigma)}\Vert,
\end{equation}

and let
\[
\mathcal L_N(t) = \frac{1}{N} 
\|\psi_k(t) - \psi^{(\sigma)} \|^2.
\]
Then $\mathcal L_N$ decays exponentially fast in time.
Hence we have inequality \eqref{exactcontrol}
with $\mathcal C_0$  as defined in \eqref{c0definition} below. 
\end{lemma}

\begin{proof}
%
We have
\begin{eqnarray*}
\partial_t \mathcal L_N(t) & = & 
\frac{2}{N} 
\langle \psi_k(t) - \psi^{(\sigma)} ,\;
\psi_k'(t) \rangle_{{\mathbb R}^d}
\\
& = &
\frac{2}{N} 
\langle \psi_k(t) - \psi^{(\sigma)}, \;
\beta \left( \psi^{(\sigma)} - \psi_k(t) \right)
 \rangle_{{\mathbb R}^d}
 \\
 & = &
 - \beta \, \frac{2}{N} 
\| \psi_k(t) - \psi^{(\sigma)} 
 \|^2
 \\
 & = &
  - 2 \, \beta \,\mathcal L_N(t). 
\end{eqnarray*}
Hence we have
$\mathcal L_N(t) = \mathcal L_N(a) \,e^{-2\beta\, t}$.
Moreover, we even have
\begin{equation}\label{eq:mathL_ass}
\|\psi_k(t)  - \psi^{(\sigma)} \|
= \|\psi_k(a) -\psi^{(\sigma)} \| \,e^{-\beta\, t}.
\end{equation}
We  have  the inequality
\[
\begin{split}
\|u_k(t)\|
&\leq
\beta \left\| \psi^{(\sigma)} - \psi_k(t) \right\|
+ \frac{C_P}{N} \sum_{l=1}^N
\left(\|\psi_k(t) - \psi^{(\sigma)} \|
+ \|\psi_l(t)-  \psi^{(\sigma)} \| \right)\\
&=
(\beta + C_P) \left\|  \psi^{(\sigma)}  - \psi_k(t) \right\|
+ \frac{C_P}{N} \sum_{l=1}^N
\, \|\psi_l(t)-  \psi^{(\sigma)} \| ,
\end{split}
\]
where we used the property $(P)$ stated in Section \ref{sec:sol_meanfield}.
Hence we have
\[
\|u_k(t)\|
\leq
 e^{-\beta\, t}\,
 \left(
(\beta + C_P)
\|\psi_k(a) - \psi^{(\sigma)}\| +
 \frac{C_P}{N}
\sum_{l=1}^N \|\psi_l(a)-  \psi^{(\sigma)} \| 
\right).
\]
By property $(\Psi)$ in Section \ref{sec:sol_meanfield} we know that $\Psi(u)\leq C_\Psi \| u \|$,
therefore we can write
\[
\Psi(u_k(t))
\leq
e^{-\beta\, t}\,C_\Psi \,
\left(
(\beta + C_P)
\|\psi_k(a) - \psi^{(\sigma)}\| +
\frac{C_P}{N}
\sum_{l=1}^N \|\psi_l(a)-  \psi^{(\sigma)} \| 
\right).
\]
Adding the term $L(\psi_k(t))$ on both sides, using \eqref{eq:L_ass} and \eqref{eq:mathL_ass}, we have
\[
\begin{split}
L(\psi_k&(t)) + \Psi(u_k(t)) \leq e^{-\beta\, t}\,C_L\,
\|\psi_k(a) - \psi^{(\sigma)}\| \, +\\
&+ e^{-\beta\, t}\,C_\Psi \,
\left(
(\beta + C_P)
\|\psi_k(a) - \psi^{(\sigma)}\| +
\frac{C_P}{N}
\sum_{l=1}^N \|\psi_l(a)-  \psi^{(\sigma)} \| 
\right).
\end{split}
\]
This yields
\[
f_N(\psi(t),\, u(t)) \leq
\Bigl(
C_L + \beta C_\Psi  + 2 C_P C_\Psi \Bigr)e^{-\beta\, t}
\frac{1}{N}\sum_{k=1}^N \|\psi_k(a) - \psi^{(\sigma)}\|.
\]
Hence 
$f_N(\psi(t),\, u(t))$ decays exponentially fast
with the rate $\beta$. For the optimal value this implies
\[
\mathcal V(N,\, a,\,b,\,\psi^0)
\leq
\Bigl(
C_L + \beta C_\Psi  + 2 C_P C_\Psi \Bigr)
\frac{1}{N \beta}\sum_{k=1}^N \|\psi_k(a) - \psi^{(\sigma)}\|.
\]
Hence
\eqref{exactcontrol}
follows
with
\begin{equation} \label{c0definition}
	\mathcal C_0= \frac{1}{ \beta} \Bigl(
	C_L + \beta C_\Psi  + 2 C_P C_\Psi \Bigr).
\end{equation}

\end{proof}

\subsection{The cheap control condition in the mean-field limit}

Also for the cheap control condition, we can compute the limit inequality in terms of measures.
Given $\mathcal C_0>0$, for all initial times
$a\geq 0$, terminal times
$b> a $ 
and  initial states
$\mu(a,x)=\mu^0(x)\in P_1(\mathbb{R}^d)$, we have

\begin{equation}\label{eq:cheap_control_ass}
	\textbf {{V}}(a,\,b,\,\mu^0)
	\leq
	\mathcal{C}_0 \int_{\mathbb{R}^d}
	\|x - \psi^{(\sigma)}\|
	\, d\mu^0( x).
\end{equation}
We recall that $\textbf {V}$ is the optimal value of
the mean-field optimization problem. To prove the mean-field cheap control inequality we follow the same idea of the microscopic case, namely we consider a stabilizing feedback law that leads to exponential 
decay of the mean-field running cost. Combining \eqref{eq:meso_dyn} and \eqref{eq_cheap_control}, and letting $N\rightarrow\infty$
we have
\[
\partial_t \mu(t,x) + \partial_x \Bigl(\beta \left(  \psi^{(\sigma)} - x \right) \mu(t,x) \Bigr) = 0.
\]

\begin{lemma}
	Consider the additional local assumption on bounded level sets of $L$ in Eq.\eqref{eq:L_ass},
	then the cheap control condition holds also for the mean-field limit problem, that is \eqref{eq:cheap_control_ass} with $\mathcal{C}_0$ the same constant \eqref{c0definition} as in the microscopic case.
\end{lemma}

\begin{proof}
Considering the mean-field formulation of \eqref{eq_cheap_control}, we obtain
\begin{equation*}
	u(t,x) \mu(t,x)= 
	\beta \left(  \psi^{(\sigma)} - x \right) \mu(t,x)
	- \mathcal{P}[\mu](t,x) \mu(t,x).
\end{equation*}
Thanks to property $(P)$ in Section \ref{sec:sol_meanfield}, this yields to
\begin{align*}
\| u(t,x)\| \mu(t,x)
& \leq
\beta \, \|  \psi^{(\sigma)} - x \| \, \mu(t,x)
+ C_P \, \mu(t,x) \int_{\mathbb{R}^d}
\left( \| y - \psi^{(\sigma)}\| + \| x- \psi^{(\sigma)}\| \right) \, d \mu(t, \, y)
\\
& \leq
(\beta + C_P) \, \|  \psi^{(\sigma)} - x \| \, \mu(t,x)
+ C_P \, \mu(t,x) \int_{\mathbb{R}^d}
\| y - \psi^{(\sigma)}\|  \, d \mu(t, \, y).
\end{align*}
We know that $\Psi(u) \leq C_\Psi \|u\|$ (property $(\Psi)$, Section \ref{sec:sol_meanfield}), hence we have
\[
	\Psi(u(t,x)) \mu(t,x) 
	\]
\[
	\leq
	\\
	C_\Psi (\beta + C_P) \, \|  \psi^{(\sigma)} - x \| \, \mu(t,x)
	+ C_\Psi C_P \, \mu(t,x) \int_{\mathbb{R}^d}
	\| y - \psi^{(\sigma)}\|  \, d \mu(t, \, y).
\]
As in the proof of Lemma \ref{lem:micro}, we can add $L(\psi_k(t))$ and integrate over in $d\mu(t,x)$ on both sides. Then using \eqref{eq:L_ass} and \eqref{eq:mathL_ass}, we obtain that the function $f$ defined in Eq. \eqref{eq:f_meso} satisfies
\begin{align*}
f(\mu(t,x), u(t,x))
 \leq
\Bigl(
C_L + \beta C_\Psi  + 2 C_P C_\Psi \Bigr)e^{-\beta\, t} \,  \int_{\mathbb{R}^d} \| x- \psi^{(\sigma)} \| \, d \mu(a,x).
\end{align*}
	For the optimal value we obtain 
	\[
	{\textbf{V}}(a,\,b,\,\mu^0)
	\leq
\mathcal{C}_0\,  \int_{\mathbb{R}^d} \| x- \psi^{(\sigma)} \|\, d \mu(a,x).
	\]
\end{proof}


\section{On the turnpike property with interior decay}\label{turnpike}

In this section we investigate whether the optimal control
problems that we discuss satisfy a turnpike property with interior decay as discussed in
\cite{zbMATH07364550}.
We start with the $N$-particle problem and
then proceed with the mean-field limit problem.

\subsection{The turnpike inequality
	for the microscopic problem}
In this section,
we present a turnpike property for the optimal control problem
$\mathcal Q(N, \, a,\, b , \, \psi^0)$
that follows from  
the dissipativity inequality
(\ref{dissipationinequalityN})
and the cheap control condition
(\ref{exactcontrol}).
As the name indicates, this property focuses
on the situation that the set where the
distance between the optimal dynamic and
the optimal static solution is small
for large $b$ is
located in final part  of the time interval $[a, b]$,
that is an interval of the form
$[ b - (1 - \lambda)(b-a), b]$ with 
$\lambda \in (0, 1)$.

We define as $\hat \psi(a,\, b, \, \psi^0)(t)$ and $\hat u(a, \, b ,\, \psi^0)(t)$
the 
{optimal}
state and 
{optimal} control respectively at time $t$ with initial state $\psi(a) = \psi^0 = \hat \psi(a,\, b, \, \psi^0)(a)$ in the interval $[a,b]$.	
Let $\lambda \in (0, \,1)$ be given.
For $b>0$ consider the number 
\begin{equation*}
	\label{zentralungleichung2}
	\mathcal A_\ast(b) := 
	{\frac{1}{N}
		\int\limits_{ {a + \lambda (b-a) }}^{  b}
		\left(\| \hat \psi(a, b,\, \psi^0)(t)  - \psi^{(\sigma)}  \|_N
		+
		\|  \hat u(a, b ,\, \psi^0)(t) -  u^{(\sigma)}\|_N
		\right)^2 \, dt.
	}
\end{equation*}
The following theorem states that the optimal control problem with $N$ agents has 
a turnpike property with interior decay:
\begin{theorem}
	\label{odethm}
	The optimization problem
	$\mathcal Q(N, \, a,\, b , \, \psi^0)$
	has a turnpike property with interior decay in the sense that
	\[
	\mathcal A_\ast(b)
	\leq   \frac{\mathcal C_0^2}{
		{\lambda\, (b-a)  }}
	\frac{1}{N}
		\sum_{k=1}^N \|\psi^0 - \psi^{(\sigma)}\|.
	\]
	where $\mathcal C_0$ is as in \eqref{c0definition}.
\end{theorem}
\begin{proof}
Due to \eqref{integraltp} we have
\[ \frac{1}{N}
\int\limits_{a}^{ b}
\left(\| \hat \psi(a, b,\, \psi^0)(t)  - \psi^{(\sigma)}  \|_N
+
\|  \hat u(a, b ,\, \psi^0)(t) -  u^{(\sigma)}\|_N
\right)^2 \, dt
\]
\[
\leq \frac{\mathcal C_0}{N}
\sum_{k=1}^N \|\psi^0 - \psi^{(\sigma)}\|.
\]
Hence  there exists $t_0 \in [a,\, {a + \lambda (b-a) }]$ such that
\begin{equation}	\label{aprioriungleichung}
	\begin{split}
	\frac{1}{N} &\left(\| \hat \psi(a, b,\, \psi^0)(t_0)  - \psi^{(\sigma)}  \|_N
	+
	\|  \hat u(a, b ,\, \psi^0)(t_0) -  u^{(\sigma)}\|_N
	\right)^2 
	\\
	&\leq
	\frac{1}{{\lambda (b - a)}}
	\frac{\mathcal C_0}{N}
\sum_{k=1}^N \|\psi^0 - \psi^{(\sigma)}\|.
		\end{split}
\end{equation}
The cheap control assumption implies  that for the optimization problem
$\mathcal Q(N, \, t_0, \,  b, \, \hat \psi^0)$
that starts
at $t_0$ with the initial state
$ \hat \psi^0
=
\hat \psi(a, b,\, \psi^0)(t_0)$
we have
\begin{equation}
	\label{exactcontrolplpl}
	\mathcal V(N,\, t_0,\,b ,\, \hat \psi^0)
	\leq
	\mathcal C_0 \,
	\frac{1}{N} 
	\, 
\sum_{k=1}^N
	\| \hat \psi^0(a, b,\, \psi^0)(t_0) - \psi^{(\sigma)}\|
	.
\end{equation}
With \eqref{aprioriungleichung} this implies
\begin{equation}
	\mathcal V(N,\, t_0,\,b ,\, \hat \psi^0)
	\leq
	\frac{\mathcal C_0^2}{{\lambda (b-a)}}
	\frac{1}{N}
\sum_{k=1}^N \|\psi^0 - \psi^{(\sigma)}\|
	.
\end{equation}
Due to 
\eqref{dissipationinequalityN}
this yields
\begin{align*}
	\mathcal A_\ast(b) &\leq 
	{\frac{1}{N}
		\int\limits_{ t_0 }^{  b}
		\left(\| \hat \psi(a, b,\, \psi^0)(t)  - \psi^{(\sigma)}  \|_N
		+
		\|  \hat u(a, b ,\, \psi^0)(t) -  u^{(\sigma)}\|_N
		\right)^2 \, dt
	}\\
&\leq
\mathcal V(N,\, t_0,\,b ,\, \hat \psi^0)
\\
&\leq   \frac{\mathcal C_0^2}{
	{\lambda (b-a)}}
\frac{1}{N}
\sum_{k=1}^N \|\psi^0 - \psi^{(\sigma)}\|.
\end{align*}
Hence we have proved the theorem.
\end{proof}
We have proved that on  the interval 
{$[a + \lambda(b-a), \, b]$,} 
the contribution to the objective functional
of this part of the time interval 
decays with 
${\cal O}
{\left( \tfrac{1}{\lambda (b - a)} \right)}$.

\subsubsection{Inductive refinement}
Consider now the following statement
\begin{theorem}
	\label{odethm1}
	Let $\alpha \in (0, \, 1)$ be given. The optimization problem
	$\mathcal Q(N, \, a,\, b , \, \psi^0)$
	has a turnpike property with interior decay in the sense that
	\[
		\frac{1}{N}
		\int\limits_{ {a + ( 1 - \alpha^2) (b-a) }}^{  b}
		\left(\| \hat \psi(a, b,\, \psi^0)(t)  - \psi^{(\sigma)}  \|_N
		+
		\|  \hat u(a, b ,\, \psi^0)(t) -  u^{(\sigma)}\|_N
		\right)^2 \, dt
	\]
	\[
	\leq   \frac{\mathcal C_0^3}{
		\alpha \, (1-\alpha)^2(b - a)^2}
	\frac{1}{N}
	\sum_{k=1}^N \|\psi^0 - \psi^{(\sigma)}\|
	\]
	where $\mathcal C_0$ is as in \eqref{c0definition}.
\end{theorem}

\begin{proof}
{
We have shown that for $\alpha \in (0, \, 1)$
in  the intervals 
\[
[b - \alpha (b-a) ,\, b]
= [ a + ( 1 - \alpha) (b-a),\, b]
\]
the 
distance between the static and the dynamic solutions
decays with the order ${\cal O}\left( \tfrac{1}{{(1-\alpha )( b-a)} } \right)$. 
}
\\
We assume now that $t_0\in [a, \,a + ( 1 - \alpha) (b-a) ]$
has been chosen as in the previous section such
that we have
\begin{equation}
	\begin{split}
	\label{aprioriungleichunginduktionsanfang}
	\frac{1}{N} &\left(\| \hat \psi(a, b,\, \psi^0)(t_0)  - \psi^{(\sigma)}  \|_N
	+
	\|  \hat u(a, b ,\, \psi^0)(t_0) -  u^{(\sigma)}\|_N
	\right)^2 \\
	&\leq
	\frac{1}{   ( 1 - \alpha) (b-a)  }
	\frac{\mathcal C_0}{N}
	\sum_{k=1}^N \|\psi^0 - \psi^{(\sigma)}\|.
		\end{split}
\end{equation}
Then as in the previous section using  the 
dissipativity inequality 
\eqref{dissipationinequalityN}
and the cheap control condition
\eqref{exactcontrol}
we obtain
\begin{equation*}
	\begin{split}
	\frac{1}{N} &\int\limits_{ t_0 }^{  b}
		\left(\| \hat \psi(a, b,\, \psi^0)(t)  - \psi^{(\sigma)}  \|_N
		+
		\|  \hat u(a, b ,\, \psi^0)(t) -  u^{(\sigma)}\|_N
		\right)^2 \, dt
	\\
& \leq 		\mathcal V(N,\, t_0,\,b ,\, \hat \psi^0)
	\\
&\leq  \frac{\mathcal C_0^2}{  ( 1 - \alpha) (b-a) }
\frac{1}{N}
\sum_{k=1}^N \|\psi^0 - \psi^{(\sigma)}\|.
	\end{split}
\end{equation*}
{
Similarly to \eqref{aprioriungleichung} we find that 
due to \eqref{dissipationinequalityN}
there exists
\[
t_1 \in [a + (1 - \alpha)(b - a), \,
a + (1 - \alpha^2)(b - a)
]
=
[ b - \alpha(b-a),\, b - \alpha^2(b-a)]
\]
}
such that
\begin{equation}
	\label{aprioriungleichunginduktion}
	\begin{split}
	\frac{1}{N} & \left(\| \hat \psi(a, b,\, \psi^0)(t_1)  - \psi^{(\sigma)}  \|_N
	+
	\|  \hat u(a, b ,\, \psi^0)(t_1) -  u^{(\sigma)}\|_N
	\right)^2
	\\
	&\leq \frac{1}{\alpha  (1-\alpha) (b-a)}
		\mathcal V(N,\, t_0,\,b ,\, \hat \psi^0)
	\\
	& \leq \frac{1}{\alpha  (1-\alpha) (b-a)}
 \frac{\mathcal C_0^2}{  ( 1 - \alpha) (b-a) }
\frac{1}{N}
\sum_{k=1}^N \|\psi^0 - \psi^{(\sigma)}\|.
.
		\end{split}
\end{equation}
Inequality \eqref{exactcontrol}
from the cheap control assumption implies  that for the optimization problem
$\mathcal Q(N, \, t_1, \,  b, \, \hat \psi^1)$
that starts
at $t_1$ with the initial state
$\hat \psi^1 = \hat \psi(a, b,\, \psi^0)(t_1)$
we have
\begin{equation}
	\label{exactcontrolplplpl1}
	\mathcal V(N,\, t_1,\,b ,\, \hat \psi^1)
	\leq
	\mathcal C_0 \,
	\frac{1}{N} \sum_{k=1}^N 
	\, 
	\| \hat \psi^1(a, b,\, \psi^0)(t_1) - \psi^{(\sigma)}\|
	.
\end{equation}
With \eqref{aprioriungleichunginduktion} this implies
\begin{equation}
	\mathcal V(N,\, t_1,\,b ,\, \hat \psi^1)
	\leq
\frac{1}{\alpha  (1-\alpha) (b-a)}
 \frac{\mathcal C_0^3}{  ( 1 - \alpha) (b-a) }
	\frac{1}{N}
	\sum_{k=1}^N \|\psi^0 - \psi^{(\sigma)}\|
	.
\end{equation}
Due to the 
dissipativity inequality 
\eqref{dissipationinequalityN}
this yields

\begin{equation*}
	\begin{split}
\frac{1}{N}&
\int\limits_{ a + (1 - \alpha^2)(b - a) }^{  b}
\left(\| \hat \psi(a, b,\, \psi^0)(t)  - \psi^{(\sigma)}  \|_N
+
\|  \hat u(a, b ,\, \psi^0)(t) -  u^{(\sigma)}\|_N
\right)^2 \, dt
\\
&\leq 
\frac{1}{N}
\int\limits_{ t_1 }^{  b}
\left(\| \hat \psi(a, b,\, \psi^0)(t)  - \psi^{(\sigma)}  \|_N
+
\|  \hat u(a, b ,\, \psi^0)(t) -  u^{(\sigma)}\|_N
\right)^2 \, ds
\\
&\leq
\mathcal V(N,\, t_1,\,b ,\, \hat \psi^1)
\\
&\leq  
\frac{1}{\alpha  (1-\alpha)^2 (b-a)^2}
{\mathcal C_0^3}
\frac{1}{N}
\sum_{k=1}^N \|\psi^0 - \psi^{(\sigma)}\|.
	\end{split}
\end{equation*}
This ends the proof.
\end{proof}
{
Hence on  
$[ a + (1 - \alpha^2)(b - a) , b]$
the contribution to the objective functional
of this part of the time interval 
decays with 
${\cal O}\left( \tfrac{1}{\alpha  (1-\alpha)^2 (b-a)^2}
\right)$.
Now we can proceed inductively to obtain a
decay of the order  ${\cal O}\left( 1/(b - a)^n \right)$
for $n\in \{1,2,3,\, ....\}$
with corresponding constants
$\mathcal C_n$ that grow with $n$.
It is also possible to state Theorem \ref{odethm1}
as an upper bound for an integral
where the lower bound of the integration interval
grows more slowly than linear,
namely only with the order $\sqrt{b-a}$.
For $b> 1$ define
\[ \mathcal B_\ast(b) = 
\frac{1}{N}
\int\limits_{ a  + 2 \sqrt{b-a}- 1}^{  b}
\left(\| \hat \psi(a, b,\, \psi^0)(t)  - \psi^{(\sigma)}  \|_N
+
\|  \hat u(a, b ,\, y_0)(t) -  u^{(\sigma)}\|_N
\right)^2 \, dt.
\]
Then we have the following result.
}
\begin{theorem}
	\label{odethm1a}
	For $b> 1$ the optimization problem
	$\mathcal Q(N, \, a,\, b , \, \psi^0)$
	has a turnpike property with interior decay in the sense that
	\[
	\mathcal B_\ast(b) 
	\leq 
	\frac{1}{
	\sqrt{b - a}(\sqrt{b - a} - 1)
	}
		\, { \mathcal C_0^3 } \,
	\frac{1}{N}
	\| \psi^0 - \psi^{(\sigma)}\|_N  
	\]
	where $\mathcal C_0$ is as in \eqref{c0definition}.
\end{theorem}
\begin{proof}
Set $\alpha =  1 -\frac{1}{\sqrt{b-a}} \in (0, \, 1)$.
Then $(1 -  \alpha)^2 = \frac{1}{b - a}$.
Hence we have
\[\alpha \,  (1 -  \alpha)^2\, (b -a)^2
=
\sqrt{b - a}(\sqrt{b - a} - 1).
\]
Since
$1 - \alpha^2 = 
\frac{ 2 \sqrt{b - a} - 1}{b - a}
$
 the assertion follows from Theorem \ref{odethm1}.
\end{proof}

\begin{remark}
	Similarly as in \cite{zbMATH07364550} we can sharpen this bound inductively.
\end{remark}

\subsection{The turnpike inequality in the mean-field limit}

In this section we state and prove the theorem for the turnpike property in term of measures, in order to do that, we use the mean-field version of the strict dissipativity \eqref{eq:strict_diss_meso} and cheap control \eqref{eq:cheap_control_ass} conditions.

\begin{theorem}	\label{thm:turnpike_meso}
	Let $\lambda \in (0, \,1)$ be given, and the interval $[a,b]$ with $b>0$.
	Consider the quantity
	\[
	\textbf A_\ast(b) =  \int_{a + \lambda (b-a)}^b \int_{\mathbb{R}^d}  \left( \| x  - \psi^{(\sigma)}  \|^2
	+
	\|  \hat u_{(a, b ,\mu^0)}(t,x) -  u^{(\sigma)}\|^2
	\right)d \hat \mu_{(a, b ,\mu^0)}(t,x) \, dt,
	\]
	where we define as $ \hat \mu_{(a, b ,\mu^0)}(t,x)$ and $\hat u_{(a, b ,\mu^0)}(t,x)$
	the density and control respectively at time $t$ with initial condition $\mu(a,x) = \mu^0(x) = \hat  \mu_{(a, b ,\mu^0)}(a,x)$.
	Then the optimization problem
	$\textbf Q(a,\, b , \, \mu^0)$
	has a turnpike property with interior decay in the sense that
	\[
	\textbf A_\ast(b) \leq   \frac{\mathcal C_0^2}{
		\lambda (b-a)}
	\int_{\mathbb{R}^d} \| x- \psi^{(\sigma)} \|\, d \mu(a,x).
	\]
	where $\mathcal C_0$ is as in \eqref{c0definition}.
\end{theorem}
\begin{proof}
	From the mean-field strict dissipativity \eqref{eq:strict_diss_meso} and cheap control \eqref{eq:cheap_control_ass} conditions, we know that the optimal density and control satisfy
		\begin{equation}
		\label{eq:integraltp_meso}
		\begin{split}
	\int_a^b \int_{\mathbb{R}^d} &
\left( \|  x - \psi^{(\sigma)}\|^2
+
\|  u(t,x)  - u^{(\sigma)} \|^2
\right) d \mu(t,x)
\, dt
\\
&\leq 
{ \mathcal C_0}	\int_{\mathbb{R}^d} \| x- \psi^{(\sigma)} \|\, d \mu(a,x).
		\end{split}
\end{equation}
	Furthermore we can write
	\[ 
	\begin{split}
	\int_a^b \int_{\mathbb{R}^d} &
		\left( \|  x - \psi^{(\sigma)}\|^2
		+
		\|  \hat u_{(a, b ,\mu^0)}(t,x)  - u^{(\sigma)} \|^2
		\right) d \hat \mu_{(a, b ,\mu^0)}(t,x) \, dt
		\\
		&\leq 
		{ \mathcal C_0}	\int_{\mathbb{R}^d} \| x- \psi^{(\sigma)} \|\, d  \mu(a,x).
	\end{split}
		\]
	Hence there exists $t_0 \in [a,\, a + \lambda (b- a) ]$ such that
		\begin{equation}
		\label{eq:aprioriungleichung_meso}
	\begin{split}
	\int_{\mathbb{R}^d} &
		\left( \|  x - \psi^{(\sigma)}\|^2
		+
		\|  \hat u_{(a, b ,\mu^0)}(t_0,x)  - u^{(\sigma)} \|^2
		\right) d \hat \mu_{(a, b ,\mu^0)}(t_0,x) \, dt
		\\
		&\leq 
		\frac{ \mathcal C_0}{\lambda( b -a) }	\int_{\mathbb{R}^d} \| x- \psi^{(\sigma)} \|\, d  \mu(a,x).
	\end{split}
		\end{equation}
	Thanks to the cheap control assumption \eqref{eq:cheap_control_ass}, the optimization problem
	$\textbf Q(t_0, \,  b, \, \hat \mu^0)$
	that starts
	at $t_0$ with the initial density
	$\hat \mu_{(a, b,\, \mu^0)}(t_0,x)$,
	satisfies
	\begin{equation*}
		\textbf V(t_0,\,b ,\, \hat \mu^0)
		\leq
		{ \mathcal C_0}\int_{\mathbb{R}^d} \| x- \psi^{(\sigma)} \|\, d  \hat \mu_{(a, b,\, \mu^0)}(t_0,x)
		.
	\end{equation*}
	Together with \eqref{eq:aprioriungleichung_meso}, we obtain
	\begin{equation*}
		\textbf V(t_0,\,b ,\, \hat \mu^0)
		\leq
		\frac{  \mathcal C_0^2}{\lambda (b -a) }	\int_{\mathbb{R}^d} \| x- \psi^{(\sigma)} \|\, d  \mu(a,x)
		.
	\end{equation*}
	Due to the dissipativity inequality \eqref{eq:strict_diss_meso}
	this yields
	\begin{align*}
		\textbf A_\ast(b) &\leq 
		\int_{t_0}^b \int_{\mathbb{R}^d}
		\left( \|  x - \psi^{(\sigma)}\|^2
		+
		\|  \hat u_{(a, b ,\mu^0)}(t,x)  - u^{(\sigma)} \|^2
		\right) d \hat \mu_{(a, b ,\mu^0)}(t,x) \, dt
		\\
		&\leq
		\textbf V( t_0,\,b ,\, \hat \mu^0)
		\\
		&\leq  \frac{\mathcal C_0^2}{
			\lambda (b-a)}	\int_{\mathbb{R}^d} \| x- \psi^{(\sigma)} \|\, d  \mu(a,x),
	\end{align*}
	that is the inequality  stated in the theorem.
\end{proof}
\begin{remark}
	The bound we have in Theorem \ref{thm:turnpike_meso}, on the distance between the optimal dynamic and
	the optimal static solution, can be inductively sharpened, using the same procedure of the microscopic case.
\end{remark}


\section{Conclusion}
	We have obtained a turnpike theorem for
	microscopic and mesoscopic optimal control problems that 
	satisfy a strict dissipativity inequality and we used these properties to show that the optimal control problems fulfill the cheap control result and in turn a turnpike property.
	Providing suitable assumptions to guarantee the existence of solutions in the mean-field limit, 
	we have proven the turnpike property  both on the level of finitely many interacting particles and the mean-field limit. The turnpike property holds true without additional assumptions. 
	Possible future work includes the numerical simulation and the extension e.g. 
	to the  
	case
	that the microscopic model is governed by a second-order dynamics.


\subsection*{Acknowledgments}
The authors thank the Deutsche Forschungsgemeinschaft (DFG, German Research Foundation) for the financial support through 320021702/GRK2326,  333849990/IRTG-2379, B04, B05 and B06 of 442047500/SFB1481, HE5386/18-1,19-2,22-1,23-1,25-1, ERS SFDdM035 and under Germany’s Excellence Strategy EXC-2023 Internet of Production 390621612 and under the Excellence Strategy of the Federal Government and the Länder. Support through the EU DATAHYKING is also acknowledged.  This work was also funded by the DFG, TRR 154, \emph{Mathematical Modelling, Simulation and Optimization Using the Example of Gas Networks}, project C03 and C05, Projektnr. 239904186.

\bibliographystyle{siam}

\bibliography{meanfield,biblioCS2}

Competing interests: The authors declare none.

\end{document}